\documentclass[11pt, a4paper, fleqn]{article}
\usepackage[T1]{fontenc}
\usepackage[utf8]{inputenc}
\usepackage{geometry}
\usepackage{amsmath,amssymb,amsthm}
\usepackage{enumitem}
\usepackage{graphicx}
\usepackage{subcaption}
\usepackage[unicode=true,
 bookmarks=false,
 breaklinks=false,hidelinks]
 {hyperref}

\numberwithin{equation}{section}
\numberwithin{figure}{section}

\theoremstyle{plain}
\newtheorem{theorem}{Theorem}[section]
\newtheorem{lemma}[theorem]{Lemma}
\newtheorem{corollary}[theorem]{Corollary}

\theoremstyle{definition}
\newtheorem{definition}[theorem]{Definition}
\newtheorem{remark}[theorem]{Remark}

\newcommand{\dotcup}{\mathrel{\dot{\cup}}}
\newcommand{\dotsub}{\mathrel{\dot{\setminus}}}

\newcommand{\Par}{\operatorname{par}}
\newcommand{\cld}{\operatorname{cld}}
\newcommand{\lev}{\operatorname{level}}
\newcommand{\PTG}{\operatorname{PTG}}
\newcommand{\evn}{\operatorname{even}}
\newcommand{\odt}{\operatorname{odd}}
\newcommand{\net}{\operatorname{net}}
\newcommand{\CC}{\mathbb{C}}
\newcommand{\FF}{\mathbb{F}}

\title{Refutation of the Non-Cancelling Intersections Conjecture}
\author{Hermann Wilhelm \\ Technische Universität Ilmenau}
\date{\today}

\begin{document}

\maketitle

\begin{abstract}
  The Non-Cancelling Intersections (NCI) conjecture of Amarilli, Monet and Suciu~\cite{amarilli2024non} states that the union of a finite family of sets can always be built from its algebraically non-cancelling intersections using only disjoint unions and subset complements. In~\cite{wilhelm2026left} the conjecture was shown to fail when the witnessing dot-algebra expression is required to be left-linear. Here we remove that restriction and show that the conjecture is false in general: there is a finite lattice for which the set associated with its top element admits no dot-algebra representation whatsoever. The counterexample is a lattice $P_{p,\mathfrak m}$ as in~\cite{wilhelm2026left}, and the argument differs in only two ways. First, we replace the sequential ``toggle game'' of~\cite{wilhelm2026left} by a corresponding tree-shaped object, the \emph{plane da-tree}, which stands to dot-algebra trees as the toggle game stands to left-linear ones. Second, we use a marked plane in which there is no admissible set of \emph{any} size between $2p$ and $4p$, which also removes the need for the Erd\H{o}s--Beck theorem and for the arithmetic Nullstellensatz. Consequently $p$ need not be astronomically large: every prime $p\ge 10^{5}$ works.
\end{abstract}

\section{Introduction}\label{sec:introduction}

This note explains how to generalize the refutation of the Non-Cancelling Intersections (NCI)
conjecture for left-linear trees to the general case. The setting and all notation are those
of~\cite{wilhelm2026left}, recalled in Section~\ref{sec:preliminaries}. Most of the proof carries over
from toggle sequences to da-trees easily. Only two arguments have to change.

\begin{enumerate}[leftmargin=2em]
  \item \textbf{The two continuity arguments.} A winning plane toggle sequence changes the number of
  red points by exactly one per move, so it visits a state of every cardinality (Step~1 of
  \cite[Proposition~7.3]{wilhelm2026left}); and the number $f(R)=|R|-\max_\ell |R\cap\ell|$ of red
  points off the main line likewise changes by at most one per move, which is how the degenerate case
  of an almost collinear state is handled (Step~3 there). Both are properties of a \emph{sequence} and
  are simply false for trees, whose node states can jump.

  The substitute is Lemma~\ref{lem:medium} below: in a tree whose leaves carry at most one point each
  and whose root carries all $p^{2}$ points, for every threshold $N\le p^{2}$ there is a node carrying
  between $N$ and $2N-2$ points. This is weaker than visiting every cardinality, and it is the reason
  why the counting of Section~\ref{sec:counting} has to exclude a whole \emph{range} of cardinalities
  rather than a single one.

  \item \textbf{The blocking-gadget lemma.} \cite[Lemma~4.2]{wilhelm2026left} --- two red level-$1$
  elements of a blocking gadget force a red unblocked one --- is proved by counting, for each vertex
  of the gadget, how often it has been picked so far. In a da-tree, ``so far'' can be replaced by ``in the corresponding subtree'', but a leaf may also be canceled by a later subtraction, so simply counting is not enough. The substitute is Lemma~\ref{lem:blockingGadgetTree}, which counts the leaves of a subtree with a sign attached: the sign records whether the leaf is inherited through an even or an odd number of right-hand sides of $-$-nodes. Lemma~\ref{lem:netPick} shows that this signed count decides membership in the state of a node, and after that the argument of~\cite{wilhelm2026left} goes through as before, with the signed counts in place of the pick counts.
\end{enumerate}

To resolve the first point, we work with point sets of size $\Theta(p)$, instead of point sets of constant size $n$ (as in \cite{wilhelm2026left}). This is also what makes the rest of the proof
\emph{simpler} rather than harder. Above $p$ the degenerate case of Step~3 disappears by
itself: a set $T$ of more than $p$ points of $\FF_p^{2}$ cannot be concentrated on a line, since a
line has only $p$ points, so at least $|T|-p$ points lie off any given line. Moreover, in the range
$2p \le |T| \le 4p$ the incidence input needed to run the probabilistic argument is completely
elementary. Two exact identities --- every point lies on exactly $p+1$ lines, every pair of points lies on
exactly one line --- together with Cauchy--Schwarz show that every such $T$ determines at least
$(p+1)^{2}/12$ traces of size at most $47$ (Lemma~\ref{lem:shortTracesCS}). Since each of these traces
has to contain one of the $w=\lceil\sqrt{2p}\rceil+1$ points marked for its line, a first-moment
computation produces a marking for which no set of size between $2p$ and $4p$ is admissible
(Lemma~\ref{lem:markingExists}). Neither the Erd\H{o}s--Beck theorem nor the transfer from $\CC$ to
$\FF_p$ is used, and consequently the astronomical bound of~\cite[\S8]{wilhelm2026left} disappears:
every prime $p\ge 10^{5}$ works.

Combining the two halves gives the main result.

\begin{theorem}\label{thm:mainIntro}
For every prime $p\ge 10^{5}$ there is a marking $\mathfrak m$ such that the lattice $P_{p,\mathfrak m}$
admits no winning dot-algebra tree. Consequently the NCI conjecture is false.
\end{theorem}

Theorem~\ref{thm:mainIntro} is proved as Theorem~\ref{thm:noWinningDaTree} and
Corollary~\ref{cor:nciFails} in Section~\ref{sec:main}. For intuition it may be helpful to understand
the left-linear case and toggle sequences first, since they are easier to picture and the proof
techniques are the same.

\subsection*{Organization}

Section~\ref{sec:preliminaries} recalls what is needed from~\cite{wilhelm2026left}.
Section~\ref{sec:daTrees} defines dot-algebra trees and states the NCI conjecture in the lattice
formulation. Section~\ref{sec:planeTrees} introduces winning plane da-trees and proves the medium-node
lemma. Section~\ref{sec:GadgetTree} proves the tree version of the blocking-gadget lemma, and
Section~\ref{sec:daToPlane} deduces that a winning dot-algebra tree for $P_{p,\mathfrak m}$ induces a
winning plane da-tree. Section~\ref{sec:counting} produces the marking, and Section~\ref{sec:main}
assembles the proof; Remark~\ref{rem:leftlinear} there records that the same counting also simplifies
the main result of~\cite{wilhelm2026left} considerably. Section~\ref{sec:open} lists two open
problems.

\subsection*{Note added}

Alexander Walz (private communication, 26 August 2026) has independently obtained a result
equivalent to Corollary~\ref{cor:nciFails}, by extending the left-linear-tree construction
of~\cite{wilhelm2026left} to unrestricted da-trees. Both proofs pass through the same step, the
extension of the blocking-gadget lemma of~\cite{wilhelm2026left} from toggle sequences to arbitrary
da-trees (Lemma~\ref{lem:blockingGadgetTree} here; Walz obtains it by expanding indicator functions
in the zeta basis of the lattice), which yields that the point set of every node of a da-tree is
admissible. Walz works with these point sets directly and does not introduce plane da-trees. The two
proofs diverge in the descent that follows, and hence in the way the resulting sets are excluded. The
argument given here descends to a node whose state has size comparable to~$p$
(Lemma~\ref{lem:medium}), and excludes admissible sets throughout that range by an elementary
Cauchy--Schwarz estimate on the marking. Walz descends instead to a node whose point set is contained
in two lines together with an exceptional set of size bounded by an absolute constant independent
of~$p$, and excludes such configurations by a probabilistic marking argument based on the
finite-field Erd\H{o}s--Beck estimate of \cite[Theorem~6.3]{wilhelm2026left}. This argument therefore
requires a tower-type prime, as in~\cite{wilhelm2026left}.

\section{Preliminaries}\label{sec:preliminaries}

We use the notation of~\cite{wilhelm2026left} throughout and recall here only what is needed to read
this note.

\paragraph{Posets and lattices.}
All posets are finite. For an element $v$ of a poset $P$ we write $\uparrow\! v$ for the set consisting
of $v$ and all its ancestors and $\downarrow\! v$ for the set consisting of $v$ and all its
descendants, $\Par(v)$ for the set of parents of $v$ and $\cld(v)$ for the set of its children; note
that $v\in\uparrow\! v$ and $v\in\downarrow\! v$. The \emph{level} $\lev(v)$ is the length of a longest
upward path starting at $v$, so $\lev(\top)=0$, elements of the same level are incomparable, and all
ancestors of a level-$k$ element have level less than $k$. In particular, if $\lev(v)=1$ then
$\uparrow\! v=\{v,\top\}$, and if $\lev(v)=2$ then $\uparrow\! v=\{v\}\cup\Par(v)\cup\{\top\}$. The
Möbius function of a lattice is given by $\mu(\top)=1$ and $\mu(v)=-\sum_{u>v}\mu(u)$; level-$1$
elements have $\mu=-1$, and a level-$2$ element with exactly two parents has $\mu=1$. For
$v\in P\setminus\{\top\}$ we write
\[
  S_v \;:=\; \downarrow\! v
\]
for the set associated with $v$ in the dot-algebra, and $S_\top:=\bigcup_{v\ne\top}S_v=P\setminus\{\top\}$.

\paragraph{Blocking gadgets.}
An element $x$ of a lattice $P$ \emph{induces a blocking gadget} if $\lev(x)=3$ and $\mu(x)=0$; the
gadget is the interval $Q:=[x,\top]=\uparrow\! x$, and $x$ \emph{blocks} the elements of
$\cld(\top)\cap\Par(x)$, that is, those level-$1$ ancestors of $x$ that are parents of $x$
\cite[Definition~4.1]{wilhelm2026left}. We always write $L$ for the set of level-$1$ elements of $Q$,
$C$ for the set of level-$2$ elements of $Q$ and $M\subseteq L$ for the set of level-$1$ elements not
blocked by $x$, so that $Q=\{\top\}\cup L\cup C\cup\{x\}$. Every $c\in C$ satisfies
$\Par(c)\subseteq M$: an element of $L\setminus M$ is a parent of $x$, so nothing lies strictly
between it and $x$.

\paragraph{Lines, traces and markings.}
For a prime $p$ we denote the $p^{2}+p$ lines of $\FF_p^{2}$ by $L_{i,j}$ as
in~\cite[\S2]{wilhelm2026left}. Every point lies on exactly $p+1$ lines and every pair of distinct
points lies on exactly one line. For $T\subseteq\FF_p^{2}$ and a line $\ell$ with $|T\cap\ell|\ge2$
the set $T\cap\ell$ is the \emph{trace} of $\ell$ on $T$. A \emph{marking} is a map $\mathfrak m$
assigning to every line $\ell$ a set $\mathfrak m(\ell)\subseteq\ell$ of points \emph{marked for
$\ell$}; a point may be marked for one line and unmarked for another. A set $T\subseteq\FF_p^{2}$ is
\emph{admissible} for $\mathfrak m$ if every trace of $T$ contains a point marked for the
corresponding line, i.e.\ if $T\cap\mathfrak m(\ell)\ne\emptyset$ for every line $\ell$ with
$|T\cap\ell|\ge2$.

\paragraph{The lattice $P_{p,\mathfrak m}$.}
Let $p$ be a prime, put $w:=\lceil\sqrt{2p}\rceil+1$ and let $\mathfrak m$ be a marking with
$|\mathfrak m(\ell)|=w$ for every line $\ell$. The lattice $P_{p,\mathfrak m}$
of~\cite[\S5]{wilhelm2026left} has depth $4$ and consists of $\top$; the $p^{2}$ level-$1$ elements
$a_{i,j}$, one for each point $(i,j)\in\FF_p^{2}$, which we identify with the points themselves; for
every line $L_{i,j}$ the $p-1$ level-$2$ elements $c_{i,j,z}$, each having $b_{i,j}$ as its only child
and two distinct marked points of $L_{i,j}$ as its parents, the pairs being distinct and covering all
of $\mathfrak m(L_{i,j})$; the $p^{2}+p$ level-$3$ elements $b_{i,j}$, one for each line, whose
level-$1$ ancestors are exactly the points of $L_{i,j}$; and $\bot$. Each $b_{i,j}$ satisfies
$\mu(b_{i,j})=0$ and therefore induces a blocking gadget whose set $L$ is the line $L_{i,j}$ and whose
set $M$ of unblocked level-$1$ elements is exactly $\mathfrak m(L_{i,j})$. The lattice has
$p^{3}+2p^{2}+2$ elements.

\section{Dot-algebra trees}\label{sec:daTrees}

We work with the lattice formulation of the conjecture, as in~\cite{wilhelm2026left}.

\begin{definition}\label{def:daTree}
Let $P$ be a lattice. A \emph{dot-algebra tree} (\emph{da-tree}) for $P$ is a finite rooted binary
tree $D$ in which
\begin{itemize}[leftmargin=2em]
  \item every leaf is labeled either with the symbol $\emptyset$ or with a vertex $v\in P$ such that
    $v\ne\top$ and $\mu(v)\ne0$;
  \item every internal node is labeled $+$ or $-$ and has a left and a right child.
\end{itemize}
The \emph{state} $X_u\subseteq P\setminus\{\top\}$ of each node $u$ is defined bottom-up as follows.
If $u$ is a leaf labeled $\emptyset$, then $X_u:=\emptyset$, and if $u$ is a leaf labeled $v$, then
$X_u:=S_v$. If $u$ is an internal node with left child $u_1$ and right child $u_2$, then
\begin{itemize}[leftmargin=2em]
  \item if $u$ is labeled $+$, it is required that $X_{u_1}\cap X_{u_2}=\emptyset$, and
    $X_u:=X_{u_1}\dotcup X_{u_2}$;
  \item if $u$ is labeled $-$, it is required that $X_{u_2}\subseteq X_{u_1}$, and
    $X_u:=X_{u_1}\dotsub X_{u_2}$.
\end{itemize}
The da-tree is \emph{winning} if the state of its root is $P\setminus\{\top\}$. It is
\emph{left-linear} if every right child of an internal node is a leaf.
\end{definition}

\begin{remark}
  The word ``winning'' is chosen in analogy to the toggle game. One could also call it ``complete'', or ``exhaustive''.
\end{remark}

So a leaf stands for one of the sets $S_v$, a $+$-node for a disjoint union and a $-$-node for the
complement of a subset inside a superset, and a winning da-tree is precisely a dot-algebra expression
for $S_\top=P\setminus\{\top\}$ whose leaves are the sets associated with vertices of nonzero Möbius
value. An example is given in Figure~\ref{fig:daTree}.

\begin{figure}[htbp]
  \centering
  \includegraphics[width=0.8\textwidth]{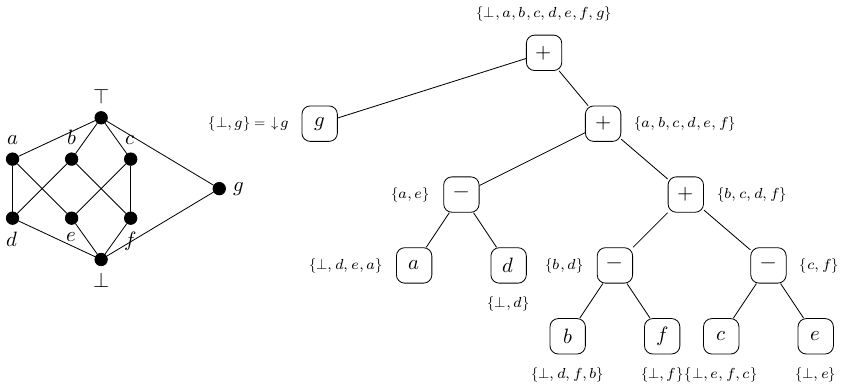}
  \caption{A lattice on the left and a corresponding winning da-tree on the right. The state of each node is written next to the node.}
  \label{fig:daTree}
\end{figure}

\begin{remark}\label{rem:emptyLeaves}
Leaves labeled $\emptyset$ are a convenience only, and admitting them makes the results below
formally stronger. They do not change what can be built: if $D$ is a da-tree whose root state is
nonempty and which has a leaf $\lambda$ labeled $\emptyset$, let $u$ be the parent of $\lambda$ (which
exists, as the root is not $\lambda$) and let $u'$ be the other child of $u$. If $u$ is labeled $+$,
or if $u$ is labeled $-$ and $\lambda$ is its right child, then $X_u=X_{u'}$ and we may replace the
subtree rooted at $u$ by the one rooted at $u'$. If $u$ is labeled $-$ and $\lambda$ is its left
child, then $X_{u'}\subseteq X_\lambda=\emptyset$, so $X_u=\emptyset$ and we may replace the subtree
rooted at $u$ by a single leaf labeled $\emptyset$. Either replacement leaves all remaining states
unchanged, preserves the requirements of Definition~\ref{def:daTree} and decreases the number of
nodes, so iterating it terminates in a da-tree with the same root state and no $\emptyset$-leaves at
all. We allow $\emptyset$-leaves because the translation of Section~\ref{sec:daToPlane} produces them:
vertices of $P_{p,\mathfrak m}$ below level $1$ carry no point of the plane. The same remark applies
verbatim to the plane da-trees of Definition~\ref{def:planeTree}.
\end{remark}

The NCI conjecture states that every finite lattice has a winning da-tree. Winning
\emph{left-linear} da-trees correspond to winning toggle sequences in the sense
of~\cite[\S3.1]{wilhelm2026left}: a toggle sequence always operates on the result obtained so far, so
all parentheses are moved to the left.

\begin{remark}\label{rem:latticeFormulation}
The conjecture of~\cite{amarilli2024non} is a statement about finite families of sets, and
Definition~\ref{def:daTree} is its lattice formulation, as in~\cite[\S3.1]{wilhelm2026left}. We recall
the translation, since it is the only link between the combinatorics below and the conjecture itself.
Let $P$ be a finite lattice in which every element is a meet of level-$1$ elements, take
$P\setminus\{\top\}$ as ground set and consider the family $\mathcal F:=\{S_a : a\in\cld(\top)\}$.
Distinct elements have distinct down-sets, and $S_u\cap S_v=\downarrow\!(u\wedge v)=S_{u\wedge v}$, so
the intersections of subfamilies of $\mathcal F$ are exactly the sets $S_v$ with $v\in P$; the empty
intersection is the ground set $S_\top=\bigcup\mathcal F$, and the lattice of intersections of
$\mathcal F$, ordered by inclusion, is $P$. The algebraically non-cancelling intersections are the
$S_v$ with $\mu(v)\ne0$. A dot-algebra representation of $S_\top$ in the sense
of~\cite[Definition~3.2]{amarilli2024non} is a term built from those sets by $\dotcup$ and
$\dotsub$, so its syntax tree \emph{is} a winning da-tree for $P$: the leaves are the $S_v$ with
$v\ne\top$ and $\mu(v)\ne0$, and the two side conditions are exactly the ones imposed in
Definition~\ref{def:daTree}. (If the representation is given as a circuit rather than as a term,
unfolding it into a tree changes neither the side conditions nor the value of any node.) Hence a
lattice with the above property and without a winning da-tree refutes the conjecture.

The lattice $P_{p,\mathfrak m}$ has that property: the $a_{i,j}$ are the level-$1$ elements,
$c_{i,j,z}$ is the meet of its two parents, $b_{i,j}$ is the meet of two unmarked points of
$L_{i,j}$ (for $p\ge11$), and $\bot$ is the meet of three points of $\FF_p^{2}$ that are not collinear; these
meets are read off the computation in~\cite[Appendix~A]{wilhelm2026left}.
\end{remark}

The main result of~\cite{wilhelm2026left} is that not every finite lattice admits a winning
left-linear da-tree. We show here that not every finite lattice admits a winning da-tree at all.

\section{Winning plane da-trees}\label{sec:planeTrees}

The plane counterpart of a da-tree is the following. It stands to the plane toggle game
$\PTG(p,\mathfrak m)$ of~\cite[\S4]{wilhelm2026left} exactly as a da-tree stands to a toggle sequence.

\begin{definition}\label{def:planeTree}
Let $p$ be a prime and $\mathfrak m$ a marking for $\FF_p^2$. A \emph{plane da-tree} for
$(p,\mathfrak m)$ is a finite rooted binary tree $\mathcal D$ in which
\begin{itemize}[leftmargin=2em]
  \item every leaf is labeled either with a point $q\in\FF_p^2$ or with the symbol $\emptyset$;
  \item every internal node is labeled $+$ or $-$ and has a left and a right child.
\end{itemize}
The \emph{state} $T_u\subseteq\FF_p^2$ of each node $u$ is defined bottom-up as follows. A leaf
labeled $q$ has state $\{q\}$ and a leaf labeled $\emptyset$ has state $\emptyset$. For an internal
node $u$ with left child $u_1$ and right child $u_2$,
\begin{itemize}[leftmargin=2em]
  \item if $u$ is labeled $+$, it is required that $T_{u_1}\cap T_{u_2}=\emptyset$, and
    $T_u:=T_{u_1}\dotcup T_{u_2}$;
  \item if $u$ is labeled $-$, it is required that $T_{u_2}\subseteq T_{u_1}$, and
    $T_u:=T_{u_1}\dotsub T_{u_2}$.
\end{itemize}
The plane da-tree is \emph{winning} if
\begin{enumerate}[label=(\roman*),leftmargin=2.6em]
  \item the state $T_u$ of every node $u$ of $\mathcal D$ is admissible for $\mathfrak m$, that is,
    every trace of $T_u$ contains a point marked for the corresponding line; and
  \item the state of the root is $\FF_p^2$.
\end{enumerate}
\end{definition}

An example of a winning plane da-tree is given in Figure~\ref{fig:plane_da_tree}.

\begin{figure}[htbp]
  \centering
  % ================= TOP ROW =================
  % Top-left: marked plane
  \begin{subfigure}[c]{0.45\textwidth}
    \includegraphics[width=\linewidth]{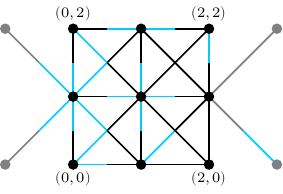}
    \caption{Marked plane}
    \label{fig:plane_da_tree_plane}
  \end{subfigure}\hfill
  % Top-right: main caption
  \begin{minipage}[c]{0.50\textwidth}
    \caption{Left: The marked plane $(3,\mathfrak m)$: for every line, the marked point is the one on the colored part of that line. Below: Plane da-tree for this marked plane. The state of every node is drawn as red points. Every state is admissible and the state of the root is $\mathbb{F}_3^2$, so the tree is winning.}
    \label{fig:plane_da_tree}
  \end{minipage}

  \vspace{1em} % space between the two rows

  % ================= BOTTOM ROW =================
  \begin{subfigure}[b]{1.0\textwidth}
    \centering
    % the counter must be set inside the subfigure environment
    \setcounter{subfigure}{1}
    \includegraphics[width=\linewidth]{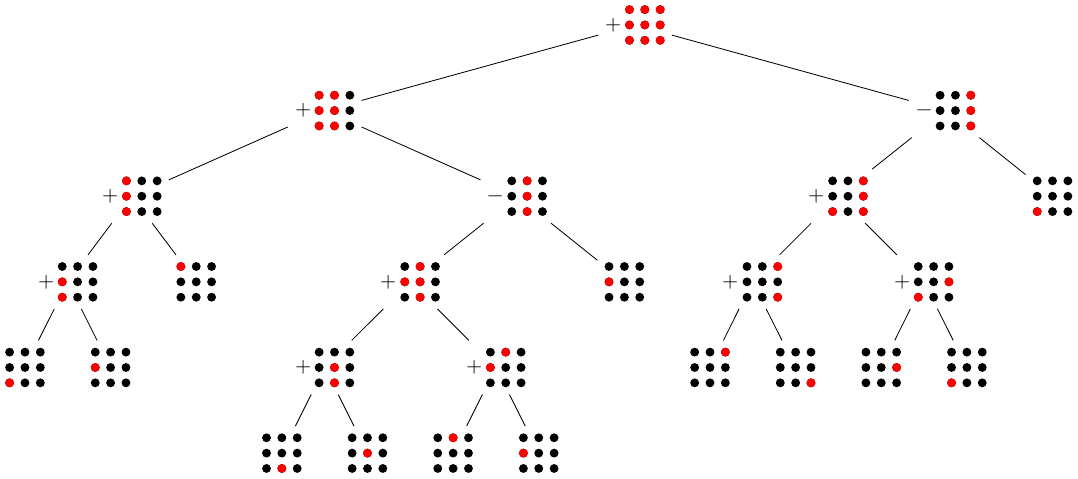}
    \caption{Winning plane da-tree.}
    \label{fig:plane_da_tree_tree}
  \end{subfigure}
\end{figure}

The following lemma is the substitute for the two continuity arguments of
\cite[Proposition~7.3]{wilhelm2026left}. It is the only place where the shape of the tree is used at all.

\begin{lemma}\label{lem:medium}
Let $p\ge2$ be a prime and let $\mathcal D$ be a winning plane da-tree for $(p,\mathfrak m)$. Then
$\mathcal D$ has a node $u$ whose state satisfies
\[
  2p \;\le\; |T_u| \;\le\; 4p-2 ,
\]
and $T_u$ is admissible for $\mathfrak m$.
\end{lemma}

\begin{proof}
We first observe that every internal node has a child carrying at least half of its points. Indeed,
let $v$ be an internal node with left child $v_1$ and right child $v_2$. If $v$ is labeled $+$, then
$|T_v|=|T_{v_1}|+|T_{v_2}|$, so $|T_{v_i}|\ge|T_v|/2$ for at least one $i$. If $v$ is labeled $-$,
then $T_v\subseteq T_{v_1}$ and hence $|T_{v_1}|\ge|T_v|\ge|T_v|/2$.

Now build a path $u_0,u_1,\dots,u_m$ in $\mathcal D$ by letting $u_0$ be the root, letting $u_{k+1}$ be
a child of $u_k$ with $|T_{u_{k+1}}|\ge|T_{u_k}|/2$ as long as $u_k$ is internal, and stopping at a
leaf $u_m$. Since $\mathcal D$ is winning we have $|T_{u_0}|=p^{2}\ge 2p$, while
$|T_{u_m}|\le 1<2p$ because $u_m$ is a leaf. Hence there is a largest index $k$ with
$|T_{u_k}|\ge 2p$, and $k<m$. Put $u:=u_k$. By maximality $|T_{u_{k+1}}|\le 2p-1$, and by the choice
of the path $|T_{u}|\le 2|T_{u_{k+1}}|\le 4p-2$. Admissibility of $T_u$ holds because $\mathcal D$ is
winning.
\end{proof}

\section{Blocking gadgets constrain every node of a da-tree}\label{sec:GadgetTree}

The aim of this section is the following tree version of \cite[Lemma~4.2]{wilhelm2026left}.

\begin{lemma}[blocking gadgets, tree version]\label{lem:blockingGadgetTree}
Let $P$ be a lattice, let $x\in P$ induce a blocking gadget $Q$, let $L$ and $C$ be the sets of
level-$1$ and of level-$2$ elements of $Q$, and let $M\subseteq L$ be the set of level-$1$ elements not
blocked by $x$. Then for every node $u$ of every da-tree $D$ for $P$ the following holds: if
$X_u\cap M=\emptyset$, then $|X_u\cap(L\cup C)|\le1$. In particular, if $|X_u\cap L|\ge2$, then
$X_u\cap M\ne\emptyset$.
\end{lemma}

In~\cite{wilhelm2026left} the corresponding statement is proved by counting, for each vertex $v$ of
the gadget, how often a toggle sequence has picked $v$ or an ancestor of $v$, and by comparing the
picks that switch $v$ on with those that switch it off. In a tree the picks are the leaves, but a leaf
does not simply switch its vertex on or off: whether it contributes positively or negatively depends
on how many times its contribution is subtracted again on the way up to $u$. This is what the
following sign records.

\begin{definition}\label{def:signedCount}
Let $D$ be a da-tree and $u$ a node of $D$. We write $D_u$ for the subtree of $D$ rooted at $u$ and
$\Lambda_u$ for the set of leaves of $D_u$. For $\lambda\in\Lambda_u$ let $r_u(\lambda)$ be the number
of nodes $v$ on the path from $u$ to $\lambda$ that are labeled $-$ and that the path leaves through
their \emph{right} child. For $b\in P$ put
\begin{align*}
  \evn(b,u) \;&:=\; \bigl|\{\lambda\in\Lambda_u:\ \lambda\text{ is labeled }b\text{ and }r_u(\lambda)\text{ is even}\}\bigr| ,\\
  \odt(b,u) \;&:=\; \bigl|\{\lambda\in\Lambda_u:\ \lambda\text{ is labeled }b\text{ and }r_u(\lambda)\text{ is odd}\}\bigr| ,\\
  \delta(b,u) \;&:=\; \evn(b,u)-\odt(b,u) ,
\end{align*}
and for $a\in P\setminus\{\top\}$ put
\[
  \net(a,u) \;:=\; \sum_{b\in\uparrow a}\delta(b,u) .
\]
\end{definition}

$\net(a,u)$ is exactly the signed number of leaves of $D_u$ whose label $b$ satisfies $a\in S_b$,
because $a\in S_b$ if and only if $b\in\uparrow a$. Note also that $\delta(\top,u)=0$ and
$\delta(b,u)=0$ for every $b$ with $\mu(b)=0$, since such vertices never label a leaf.

\begin{lemma}\label{lem:netPick}
Let $D$ be a da-tree for $P$, let $u$ be a node of $D$ and let $a\in P\setminus\{\top\}$. Then
\[
  \net(a,u)=
  \begin{cases}
    1, & a\in X_u,\\
    0, & a\notin X_u.
  \end{cases}
\]
\end{lemma}

\begin{proof}
Induction on the size of $D_u$. If $u$ is a leaf labeled $\emptyset$, then $u$ is the only leaf of
$D_u$ and it carries no label in $P$, so $\net(a,u)=0$, and indeed $a\notin X_u=\emptyset$. If $u$ is a
leaf labeled $v$, then $r_u(u)=0$ and
$\delta(b,u)=1$ for $b=v$ and $\delta(b,u)=0$ otherwise, so $\net(a,u)=1$ if $v\in\uparrow a$ and
$\net(a,u)=0$ otherwise; since $v\in\uparrow a$ if and only if $a\in\downarrow v=S_v=X_u$, this is the
assertion.

Let now $u$ be internal with left child $u_1$ and right child $u_2$, and put
\[
  \sigma:=\begin{cases} +1, & u\text{ is labeled }+,\\ -1, & u\text{ is labeled }-.\end{cases}
\]
Then $\Lambda_u=\Lambda_{u_1}\dotcup\Lambda_{u_2}$, and for $\lambda\in\Lambda_{u_1}$ we have
$r_u(\lambda)=r_{u_1}(\lambda)$, whereas for $\lambda\in\Lambda_{u_2}$ we have
$r_u(\lambda)=r_{u_2}(\lambda)+1$ if $u$ is labeled $-$ and $r_u(\lambda)=r_{u_2}(\lambda)$ if $u$ is
labeled $+$. Hence $\delta(b,u)=\delta(b,u_1)+\sigma\,\delta(b,u_2)$ for every $b$, and therefore
\[
  \net(a,u)=\net(a,u_1)+\sigma\,\net(a,u_2) .
\]
By the induction hypothesis $\net(a,u_i)$ is $1$ if $a\in X_{u_i}$ and $0$ otherwise. If $u$ is
labeled $+$, then $X_{u_1}\cap X_{u_2}=\emptyset$, so $a$ lies in at most one of the two states and
$\net(a,u_1)+\net(a,u_2)$ is $1$ if $a\in X_{u_1}\dotcup X_{u_2}=X_u$ and $0$ otherwise. If $u$ is
labeled $-$, then $X_{u_2}\subseteq X_{u_1}$, so $\net(a,u_1)-\net(a,u_2)$ is $1$ exactly if
$a\in X_{u_1}$ and $a\notin X_{u_2}$, that is, exactly if $a\in X_{u_1}\dotsub X_{u_2}=X_u$, and $0$
otherwise.
\end{proof}

\begin{proof}[Proof of Lemma~\ref{lem:blockingGadgetTree}]
Let $u$ be a node of $D$ and assume $X_u\cap M=\emptyset$. We abbreviate $\delta(b):=\delta(b,u)$.
Recall from Section~\ref{sec:preliminaries} that $Q=\{\top\}\cup L\cup C\cup\{x\}$ and that
$\Par(c)\subseteq M$ for every $c\in C$.

\emph{Level-$1$ elements.} Let $a\in L$. Since $\lev(a)=1$ we have $\uparrow a=\{a,\top\}$ and
therefore, using $\delta(\top)=0$ and Lemma~\ref{lem:netPick},
\begin{align}\label{eq:tree-level1}
  \delta(a)\;=\;\net(a,u)\;=\;[\,a\in X_u\,] .
\end{align}
In particular $\delta(a)=0$ for every $a\in M$, by the assumption $X_u\cap M=\emptyset$.

\emph{Level-$2$ elements.} Let $c\in C$. Since $\lev(c)=2$ we have
$\uparrow c=\{c\}\cup\Par(c)\cup\{\top\}$ with $\Par(c)\subseteq M$, so by \eqref{eq:tree-level1} all
parents of $c$ contribute $0$ to $\net(c,u)$, and
\begin{align}\label{eq:tree-level2}
  \net(c,u)\;=\;\delta(c)+\sum_{a\in\Par(c)}\delta(a)+\delta(\top)\;=\;\delta(c) ,
\end{align}
so that $\delta(c)=[\,c\in X_u\,]$ by Lemma~\ref{lem:netPick}.

\emph{The bottom of the gadget.} Since $\mu(x)=0$, the vertex $x$ never labels a leaf, so
$\delta(x)=0$; and $\uparrow x=Q$. Hence, by \eqref{eq:tree-level1}, \eqref{eq:tree-level2} and
Lemma~\ref{lem:netPick} applied to $x$,
\[
  [\,x\in X_u\,]\;=\;\net(x,u)\;=\;\delta(x)+\delta(\top)+\sum_{a\in L}\delta(a)+\sum_{c\in C}\delta(c)
  \;=\;\bigl|X_u\cap L\bigr|+\bigl|X_u\cap C\bigr| .
\]
The left-hand side is $0$ or $1$ and every summand on the right is $0$ or $1$, so at most one element
of $L\cup C$ lies in $X_u$. This is the first assertion, and the second is its contraposition
restricted to $L$.
\end{proof}

Note that the lemma holds for every da-tree, winning or not, and for every node, including the leaves.

\section{From da-trees to plane da-trees}\label{sec:daToPlane}

\begin{lemma}\label{lem:daToPlane}
Let $p$ be a prime and $\mathfrak m$ a marking with $|\mathfrak m(\ell)|=w$ for every line $\ell$. If
$P_{p,\mathfrak m}$ admits a winning da-tree, then there is a winning plane da-tree for $(p,\mathfrak m)$.
\end{lemma}

\begin{proof}
Let $D$ be a winning da-tree for $P_{p,\mathfrak m}$. Let $\mathcal D$ be the tree obtained from $D$ by
keeping the underlying tree and the labels $+$ and $-$ of the internal nodes, and by relabeling each
leaf as follows: a leaf labeled with a level-$1$ element $a_{i,j}$ becomes a leaf labeled with the
point $(i,j)$, and a leaf labeled $c_{i,j,z}$, $\bot$ or $\emptyset$ becomes a leaf labeled
$\emptyset$. Note that these are the only possible leaf labels, since the vertices $b_{i,j}$ have
$\mu(b_{i,j})=0$.

For a node $u$ of $D$ put
\[
  \widehat T_u \;:=\; \bigl\{\,(i,j)\in\FF_p^2 \;:\; a_{i,j}\in X_u \,\bigr\} .
\]
We claim that $\widehat T_u$ is the state $T_u$ of the corresponding node of $\mathcal D$, and that
$\mathcal D$ is a winning plane da-tree.

\emph{Leaves.} Distinct level-$1$ vertices are incomparable, so the only level-$1$ vertex in
$S_{a_{i,j}}=\downarrow\! a_{i,j}$ is $a_{i,j}$ itself, hence $\widehat T_u=\{(i,j)\}$ for a leaf
labeled $a_{i,j}$. Moreover $S_{c_{i,j,z}}=\{c_{i,j,z},b_{i,j},\bot\}$ and $S_\bot=\{\bot\}$ contain
no level-$1$ vertex, so $\widehat T_u=\emptyset$ for the remaining leaves. In both cases
$\widehat T_u=T_u$.

\emph{Internal nodes.} Let $u$ be internal with children $u_1,u_2$ and assume inductively
$\widehat T_{u_i}=T_{u_i}$. If $u$ is labeled $+$ then $X_{u_1}\cap X_{u_2}=\emptyset$, hence
$\widehat T_{u_1}\cap \widehat T_{u_2}=\emptyset$, the requirement of Definition~\ref{def:planeTree} is
met and $\widehat T_u=\widehat T_{u_1}\dotcup\widehat T_{u_2}=T_u$. If $u$ is labeled $-$ then
$X_{u_2}\subseteq X_{u_1}$, hence $\widehat T_{u_2}\subseteq\widehat T_{u_1}$, the requirement is met
and $\widehat T_u=\widehat T_{u_1}\dotsub\widehat T_{u_2}=T_u$.

\emph{Admissibility.} Let $u$ be any node of $\mathcal D$ and let $L_{i,j}$ be a line with
$|T_u\cap L_{i,j}|\ge2$. We must show that at least one point of $T_u\cap L_{i,j}$ is marked for
$L_{i,j}$.

The vertex $x:=b_{i,j}$ induces a blocking gadget in $P_{p,\mathfrak m}$. Its set $L$ of level-$1$
elements is exactly the line $L_{i,j}$, and its set $M$ of unblocked level-$1$ elements is exactly
$\mathfrak m(L_{i,j})$; see~\cite[\S5]{wilhelm2026left}. Under the identification of $a_{k,l}$ with the
point $(k,l)$ we have $X_u\cap L=T_u\cap L_{i,j}$, so $|X_u\cap L|\ge2$, and
Lemma~\ref{lem:blockingGadgetTree} yields an element of $M\cap X_u$, that is, a point of
$T_u\cap L_{i,j}\cap\mathfrak m(L_{i,j})$. Hence the trace of $L_{i,j}$ on $T_u$ contains a point
marked for $L_{i,j}$, as required.

\emph{Root.} As $D$ is winning, the state of its root is $P_{p,\mathfrak m}\setminus\{\top\}$, which
contains every $a_{i,j}$; hence the root of $\mathcal D$ has state $\FF_p^2$.
\end{proof}

\section{A marking with no admissible set of medium size}\label{sec:counting}

Throughout this section $T\subseteq\FF_p^2$ and we write
\[
  j_\ell \;:=\; |T\cap\ell|
\]
for a line $\ell$ of $\FF_p^2$. The whole section rests on the two identities
\begin{align}\label{eq:twoIdentities}
  \sum_{\ell} j_\ell \;=\; |T|\,(p+1),
  \qquad
  \sum_{\ell}\binom{j_\ell}{2} \;=\; \binom{|T|}{2},
\end{align}
where both sums range over all $p^2+p$ lines: the first because every point lies on exactly $p+1$
lines, the second because every pair of distinct points in $T$ lies on exactly one line. Together they
give
\begin{align}\label{eq:secondMoment}
  \sum_\ell j_\ell^2 \;=\; 2\sum_\ell\binom{j_\ell}{2}+\sum_\ell j_\ell
  \;=\;|T|\bigl(|T|-1\bigr)+|T|(p+1)\;=\;|T|\bigl(|T|+p\bigr).
\end{align}

\begin{lemma}\label{lem:shortTracesCS}
Let $p\ge5$ be a prime and let $T\subseteq\FF_p^2$ with $2p\le |T|\le 4p$. Then at least
$\frac{(p+1)^{2}}{12}$ lines $\ell$ satisfy $2\le j_\ell\le 47$. In particular $T$ has at least
$\frac{(p+1)^{2}}{12}$ traces of size at most $47$.
\end{lemma}

\begin{proof}
Write $t:=|T|$ and let
\[
  \mathcal N:=\{\ell : j_\ell\ge2\},\qquad N:=|\mathcal N| .
\]
There are $p^2+p$ lines in total, so at most $p^2+p$ of them satisfy $j_\ell=1$ and therefore, by the
first identity in \eqref{eq:twoIdentities},
\begin{align}\label{eq:AboundCS}
  \sum_{\ell\in\mathcal N} j_\ell \;=\; t(p+1)-\bigl|\{\ell:j_\ell=1\}\bigr|
  \;\ge\; t(p+1)-(p^2+p)\;=\;(p+1)(t-p).
\end{align}
By Cauchy--Schwarz and \eqref{eq:secondMoment},
\[
  \Bigl(\sum_{\ell\in\mathcal N} j_\ell\Bigr)^{2}
  \;\le\; N\sum_{\ell\in\mathcal N} j_\ell^{2}
  \;\le\; N\sum_{\ell} j_\ell^{2}
  \;=\; N\,t\,(t+p),
\]
so with \eqref{eq:AboundCS} and $t\ge2p$, which gives both $t-p\ge t/2$ and $t+p\le\frac32 t$,
\begin{align}\label{eq:MatLeastSixth}
  N\;\ge\;\frac{(p+1)^2(t-p)^2}{t(t+p)}\;\ge\;\frac{(p+1)^{2}\,t^{2}/4}{t\cdot\frac32 t}
  \;=\;\frac{(p+1)^2}{6}.
\end{align}

Next we bound the number of lines carrying many points. By the first identity in
\eqref{eq:twoIdentities} and $t\le4p$,
\[
  \bigl|\{\ell : j_\ell\ge48\}\bigr| \;\le\; \frac{1}{48}\sum_\ell j_\ell \;=\;\frac{t(p+1)}{48}
  \;\le\;\frac{4p(p+1)}{48}\;=\;\frac{p(p+1)}{12}\;<\;\frac{(p+1)^{2}}{12}\;\le\;\frac N 2 ,
\]
the last step by \eqref{eq:MatLeastSixth}. Consequently at least $N-N/2=N/2\ge\frac{(p+1)^{2}}{12}$
lines satisfy $2\le j_\ell\le 47$. Each of them determines a trace of size at most $47$, and distinct
lines determine distinct traces.
\end{proof}

\begin{lemma}\label{lem:markingExists}
Let $p\ge 10^{5}$ be a prime and put $w=\lceil\sqrt{2p}\rceil+1$. Then there is a marking $\mathfrak m$
with $|\mathfrak m(\ell)|=w$ for every line $\ell$ of $\FF_p^2$ such that no set $T\subseteq\FF_p^2$
with $2p\le|T|\le 4p$ is admissible for $\mathfrak m$.
\end{lemma}

\begin{proof}
Choose $\mathfrak m$ at random: for every line $\ell$ independently, let $\mathfrak m(\ell)$ be a
uniformly random $w$-element subset of $\ell$. Thus every point of $\ell$ is marked for $\ell$ with
probability $w/p$, and the markings of distinct lines are independent.

Fix $T$ with $2p\le|T|\le4p$ and let $\mathcal S$ be a set of exactly $\lceil (p+1)^{2}/12\rceil$ lines
with $2\le j_\ell\le 47$, which exists by Lemma~\ref{lem:shortTracesCS}. If $T$ is admissible, then for
every $\ell\in\mathcal S$ the trace $T\cap\ell$ contains a point marked for $\ell$, an event of
probability at most
\[
  \Pr\bigl[\,T\cap\ell\cap\mathfrak m(\ell)\ne\emptyset\,\bigr]\;\le\;\frac{j_\ell\,w}{p}\;\le\;\frac{47w}{p}
  \;<\;\frac{48w}{p}
\]
by the union bound over the at most $47$ points of the trace. Distinct traces come from distinct lines,
so these events are independent and
\begin{align}\label{eq:probAdmissible}
  \Pr[\,T\text{ is admissible}\,]\;\le\;\Bigl(\frac{48w}{p}\Bigr)^{\lceil (p+1)^{2}/12 \rceil} .
\end{align}
We simplify the base and the exponent of \eqref{eq:probAdmissible}. As $p\ge12$ we have
$w=\lceil\sqrt{2p}\rceil+1\le\sqrt{2p}+2\le2\sqrt p$ and hence
\begin{align}\label{eq:baseBound}
  \frac{48w}{p}\;\le\;\frac{96}{p^{0.5}}\;\le\;\frac{1}{p^{0.1}} ,
\end{align}
where the second inequality is $96\le p^{0.4}$, which holds because $p^{0.4}\ge(10^{5})^{0.4}=10^{2}=100$. In particular the base is at most $1$, so the exponent in
\eqref{eq:probAdmissible} may be decreased to $p^{2}/12\le\lceil (p+1)^{2}/12\rceil$, and
\begin{align}\label{eq:probAdmissibleTwo}
  \Pr[\,T\text{ is admissible}\,]\;\le\;p^{-p^{2}/120} .
\end{align}
Let $E_p$ be the expected number of admissible sets $T$ with $2p\le|T|\le4p$. There are at most
$2p+1\le p^{2}$ possible cardinalities, and for $2p\le t\le 4p\le p^{2}/2$ we have
$\binom{p^2}{t}\le \binom{p^2}{4p} \le \bigl(ep^2/(4p)\bigr)^{4p}=(ep/4)^{4p}\le p^{4p}$, so by
\eqref{eq:probAdmissibleTwo}
\[
  E_p\;\le\;\sum_{t=2p}^{4p}\binom{p^2}{t}\,p^{-p^{2}/120}
  \;\le\;p^{4p+2}\cdot p^{-p^{2}/120}\;\le\;p^{\,5p-p^{2}/120} .
\]
The exponent is negative as soon as $p>600$, so $E_p<1$. Since $E_p$ is the expectation of a
non-negative integer-valued random variable, there is a marking $\mathfrak m$ for which the number of
admissible sets $T$ with $2p\le|T|\le4p$ is $0$.
\end{proof}

\begin{remark}\label{rem:slack}
The first-moment estimate has enormous slack, and it is worth saying where it comes from. The
probability that a fixed $T$ is admissible is at most $q^{E}$, where $q\approx 96/\sqrt p$ is the
probability that one short trace is hit and $E=\lceil (p+1)^{2}/12\rceil$ is the number of short
traces available, while the number of candidate sets is only $\exp(O(p\ln p))$. Since
$\ln(1/q)\approx\tfrac12\ln p$, the cost of admissibility is $\exp(-\Theta(p^{2}\ln p))$ and beats the
number of candidates by a factor $\Theta(p)$ in the exponent. Consequently neither constant in
Lemma~\ref{lem:shortTracesCS} matters: $\Omega(p^{1+\varepsilon})$ traces would already suffice
instead of the $\Omega(p^{2})$ we have, and the cutoff $47$ could be replaced by any $k$ with
$kw/p<1$, that is, by anything up to about $\sqrt{p/2}$. This is what makes the crude union bound over
all of the $\binom{p^2}{t}$ sets sufficient, and it is also why no structural information about $T$
beyond Lemma~\ref{lem:shortTracesCS} is needed. By contrast, for sets of constant size $n$ both
quantities are of the form $p^{\Theta(1)}$ --- there are $\binom{p^{2}}{n}\le p^{2n}$ candidates and
the probability is $p^{-\Theta(n^{2})}$ --- so the comparison is between two constants in the
exponent of $p$. One then needs the value of the constant in the Erd\H{o}s--Beck bound, and one needs
the almost collinear sets to be treated separately, since for them that bound says nothing. This is
what forces the more delicate analysis of~\cite[\S6--\S7]{wilhelm2026left}.
\end{remark}

\section{The main theorem}\label{sec:main}

\begin{theorem}\label{thm:noWinningDaTree}
Let $p\ge 10^{5}$ be a prime and let $\mathfrak m$ be a marking as provided by
Lemma~\ref{lem:markingExists}. Then the lattice $P_{p,\mathfrak m}$ admits no winning da-tree.
\end{theorem}

\begin{proof}
Suppose $D$ were a winning da-tree for $P_{p,\mathfrak m}$. By Lemma~\ref{lem:daToPlane} there is a
winning plane da-tree $\mathcal D$ for $(p,\mathfrak m)$. Applying Lemma~\ref{lem:medium} to $\mathcal D$
produces a node whose state $T$ is admissible for $\mathfrak m$ and satisfies
\[
  2p\;\le\;|T|\;\le\;4p-2 .
\]
This contradicts the choice of $\mathfrak m$ in Lemma~\ref{lem:markingExists}.
\end{proof}

\begin{corollary}\label{cor:nciFails}
There is a finite lattice admitting no winning da-tree. Consequently the NCI conjecture is false.
\end{corollary}

\begin{proof}
Take $p=100\,003$ and let $\mathfrak m$ be as in Lemma~\ref{lem:markingExists}. By
Theorem~\ref{thm:noWinningDaTree} the lattice $P_{p,\mathfrak m}$f admits no winning da-tree, so by
Remark~\ref{rem:latticeFormulation} the family $\{S_{a_{i,j}} : (i,j)\in\FF_p^{2}\}$ is a finite
family of sets whose union admits no dot-algebra representation.
\end{proof}

\begin{remark}\label{rem:size}
The lattice $P_{p,\mathfrak m}$ has $p^{3}+2p^{2}+2$ elements. Taking the smallest prime allowed by
Theorem~\ref{thm:noWinningDaTree}, namely $p=100\,003$, the counterexample has
\[
  100\,003^{3}+2\cdot 100\,003^{2}+2 \;<\;  1.1\cdot 10^{15}
\]
elements. This is of course still far from a lattice one would want to write down, but it is a bound
of an entirely different nature from the $10^{10^{10^{2215}}}$
of~\cite[Theorem~8.5]{wilhelm2026left}.
\end{remark}

\begin{remark}\label{rem:leftlinear}
The above also yields a much shorter proof of the main result of~\cite{wilhelm2026left}, with a far
smaller counterexample. Indeed, a winning toggle sequence for $P_{p,\mathfrak m}$ gives a winning
plane toggle sequence by \cite[Lemma~5.1]{wilhelm2026left}, and its states change by one point per
move, so one of them has exactly $2p$ points and is admissible --- which Lemma~\ref{lem:markingExists}
forbids. In particular Sections~6 and~8 of~\cite{wilhelm2026left} --- the finite Erd\H{o}s--Beck
theorem, the passage from $\CC$ to $\FF_p$ by the Nullstellensatz, and the resulting bound --- are not
needed for that result either. This is precisely the ``direct incidence-geometric argument over
$\FF_p$'' asked for in the remark closing~\cite[\S8]{wilhelm2026left}; the point is that it is
available not in the range of constant-size point sets, but in the range $|T|\approx 3p$, where the
trivial identities \eqref{eq:twoIdentities} already force a set to determine $\Omega(p^{2})$ short
traces.
\end{remark}

\section{Open problems}\label{sec:open}

\begin{enumerate}[leftmargin=2em]
  \item \textbf{Explicit counterexamples and lower bounds.} The proof of
  Theorem~\ref{thm:noWinningDaTree} is still a first-moment argument and produces no explicit marking.
  What is the smallest lattice on which no winning da-tree exists? Exhaustive search on
  $P_{p,\mathfrak m}$ for small $p$ is now conceivable.
  \item \textbf{Quantifying $\mu=0$ vertices.} As in~\cite{wilhelm2026left}: how many vertices with
  $\mu=0$ must a weakened da-tree use, and which ones?
\end{enumerate}

\section*{Use of AI}

The concepts and proof strategies in this paper were conceived by the author. Claude (Anthropic) was used to formalize the proofs from the author's initial notes and sketch. The author has verified all mathematical content and takes full responsibility for the paper.

\section*{Acknowledgements}

The author thanks Maximilian Merz for helpful discussions.

\bibliographystyle{amsplain}
\bibliography{references}

\providecommand{\bysame}{\leavevmode\hbox to3em{\hrulefill}\thinspace}
\providecommand{\MR}{\relax\ifhmode\unskip\space\fi MR }
\providecommand{\MRhref}[2]{%
  \href{http://www.ams.org/mathscinet-getitem?mr=#1}{#2}
}
\providecommand{\href}[2]{#2}
\begin{thebibliography}{1}

\bibitem{amarilli2024non}
Antoine Amarilli, Mika{\"{e}}l Monet, and Dan Suciu, \emph{The
  {N}on-{C}ancelling {I}ntersections {C}onjecture}, arXiv preprint
  arXiv:2401.16210 (2024).

\bibitem{wilhelm2026left}
H.~Wilhelm, \emph{The non-cancelling-intersections conjecture fails for
  left-linear trees}, arXiv preprint arXiv:2608.19414 (2026).

\end{thebibliography}

\end{document}